\documentclass[12pt]{amsart}
\newtheorem{thm}{Theorem}[section]
\newtheorem{lemma}[thm]{Lemma}
\newtheorem{cor}[thm]{Corollary}
\newtheorem{prop}[thm]{Proposition}
\newtheorem{rem}[thm]{Remark}

\newtheorem{exa}[thm]{Example}

\newtheorem{dfn}[thm]{Definition}

\def\Tor{\textnormal{Tor}}
\title{On the Betti numbers of some semigroup rings}
\thanks{ The second author was supported was supported by the CNCSIS-UEFISCDI project PN II-RU PD 23/06.08.2010 and by the strategic grant POSDRU/89/1.5/S/58852, Project ``Postdoctoral program for training scientific researchers" co-financed by the European Social Fund within the Sectorial Operational Program Human Resources Development 2007 -- 2013.
}
\author{Vincenzo Micale }
\address{Dipartimento di Matematica, Universit\`a di Catania,} \email{vmicale@dmi.unict.it} 
\author{Anda Olteanu}
\address{Faculty of Mathematics and Computer Science, Ovidius University, Bd.\ Mamaia 124,
 900527 Constanta, Romania,} \email{olteanuandageorgiana@gmail.com}

\begin{document}
\maketitle
\begin{abstract}
For any numerical semigroup $S$, there are infinitely many numerical symmetric semigroups $T$ such that $S=\frac{T}{2}$ (for the definition of $\frac{T}{2}$ see below) is their half. We are studying the Betti numbers of the numerical semigroup ring $K[T]$ when $S$ is a $3$-generated numerical semigroup or telescopic. We also consider  $4$-generated symmetric semigroups and the so called $4$-irreducible numerical semigroups.
\end{abstract}

\section{Introduction}
Let $S=\langle n_1, n_2,\dots,n_k\rangle$ be a numerical semigroup. A subalgebra of $K[t]$ generated by a finite number of monomials $t^{n_1},\dots,t^{n_k}$ is called a {\em positive affine semigroup ring} associated to $S$ and it is denoted by $K[S]=K[t^{n_1},\dots,t^{n_k}]$. Then $K[S]$ is graded and for each $\alpha\in\mathbb N$, $\dim_K K[S]_{\alpha}$ is $1$ or $0$, depending of if $\alpha\in S$ or not. We can write $K[S]=K[X_1,\dots,X_k]/I=A/I$, $\deg X_i=\alpha_i$, where $I$ is generated by those binomials $X^\beta-X^\gamma$ for which $\sum_i\alpha_i\beta_i=\sum_i\alpha_i\gamma_i$. The Koszul complex on $X_1,\dots,X_k$ is a minimal resolution of $K$. Tensoring with $K[S]$ and taking the $i$-th homology gives $\Tor_i^A(K[S],K)$. This is graded with $\Tor_i^A(K[S],K)=\oplus_{s\in S}\Tor_i^A(K[S],K)_s$ and the Betti numbers of $K[S]$ are given by $\beta_{i,s}(K[S])=\dim_K\Tor_i^A(K[S],K)_s$.

Not much it is known in general about $\beta_{i,s}(K[S])$. In \cite{He} the author finds the $\beta_{i,s}(K[S])$ for $S$ a $3$-generated not symmetric and in \cite{B} the author studies the Betti numbers for the symmetric not complete intersection $4$-generated case. A good reference on the homology of $K[S]$ is \cite{F}.

In this paper we are interested on $\beta_{i,s}(K[T])$ where $T$ is any (among the infinitely many) numerical semigroups for which $S$ is its half.
Furthermore we study $\beta_{i,s}(K[S])$ for the case when $S$ is a $4$-irreducible numerical semigroup (see the definition below).

In Section 2 we give some preliminaries about numerical semigroups. In particular we define the concept of half of a semigroup and we show how to find the infinitely many numerical semigroups for which a fixed semigroup $S$ is their half (see Proposition~\ref{half}). Finally we relate the  
$\beta_{i,s}(K[S])$ to the simplicial complex $\Delta_s$ associated to $S$ (see Proposition~\ref{complex}).

In Section 3 we concentrate on the case when $S$ is telescopic and we find and compare the Betti numbers of $K[T_1]$ and $K[T_2]$, where $T_1$ and $T_2$ are any two numerical semigroups for which $S$ is their half (see Proposition~\ref{Bettitel} and Corollary~\ref{b}). We can apply the previous result to all the three generated symmetric semigroups, since all of them are telescopic (see Remark~\ref{c}).

In Section 4 we concentrate on the case when $S$ is $3$-generated not symmetric and we find and compare the Betti numbers of $K[T_1]$ and $K[T_2]$ (excluded the second Betti numbers), where $T_1$ and $T_2$ are any two numerical semigroups for which $S$ is their half (see Proposition~\ref{9}).
  
In Section 5 we consider the case of $4$-generated symmetric, but not complete intersection semigroup $S$ and we determine the total Betti numbers of $K[T]$ where $T$ is any numerical semigroup for which $S$ is its half (see Proposition~\ref{4sym}).

Finally in Section 6 we determine $\beta_{i,s}(K[S])$ when $S$ is $4$-irreducible (see Remark~\ref{4irr}).

\section{Preliminaries}
We start this section by recalling some well known facts on numerical
semigroups and semigroup rings. For more details see \cite{BDF} and \cite{RGS} for instance.

A subsemigroup $S$ of the monoid of natural numbers $(\mathbb
N,+)$, such that $0\in S$, is called a {\em numerical semigroup}.
Each numerical semigroup $S$ has a natural partial ordering $\leq
_S$ where, for every $s$ and $t$ in $S$, $s\leq _S t$ if there is
an $u\in S$ such that $t=s+u$. The set $\{n_i\}$ of the minimal
elements in $S\setminus \{0\}$ in this ordering is called {\em the
minimal set of generators} for $S$. In fact all elements of $S$
are linear combinations of minimal elements, with non-negative
integers coefficients. Note that the minimal set $\{n_i\}$ of
generators is finite since for any $s\in S$, $s\neq 0$, we have
that $n_i$ is not congruent to $n_j$ modulo $s$.

A numerical semigroup $S$ generated by $n_1<n_2<\dots<n_k$ is denoted
by $\langle n_1, n_2,\dots,n_k\rangle$. Since $\langle n_1,n_2,\dots,n_k\rangle$ is isomorphic to $\langle
dn_1,dn_2,\dots,dn_k\rangle$ for any $d\in \mathbb N\setminus
\{0\}$, we assume, in the sequel, that $\gcd
(n_1,n_2,\dots,n_k)=1$. It is well known that this condition is
equivalent to $| \mathbb N\setminus S |<\infty$. Hence there is a
well defined integer $g(S)=\max \{x\in \mathbb Z\mid x\notin
S\}$, called the {\em Frobenius number} of $S$.

We denote by $T(S)$ the set $\{x\in\mathbb
Z\setminus S\mid x+s\in S\ \mbox{for every}\ s\in
S\setminus\{0\}\}$, sometimes called the {\em set of pseudo-Frobenius numbers}. 
Of course $g(S)\in T(S)$ for every $S$.
The cardinality of $T(S)$ is the {\em type} $t(S)$ of $S$. A numerical semigroup $S$ is called {\em symmetric}
if $T(S)=\{g(S)\}$, that is $S$ is symmetric if and only if its type is one.

Set $\frac{S}{2}=\{x\in \mathbb N\ |\ 2x\in S\}$. It easy to see that $\frac{S}{2}$ is a semigroup containing $S$, called the {\em half} of $S$.

\begin{prop}\cite[Corollary 6.8]{RGS}\label{half}
Let $S$ be a numerical semigroup. Then there exist infinitely many symmetric numerical semigroups $T$ such that $S=\frac{T}{2}$.
\end{prop}

\medskip

In the proofs of \cite[Theorem 6.7]{RGS}) and of last proposition it is explained how to choose the infinitely many $T$ for every $S=\langle n_1, n_2,\dots,n_k\rangle$. If $T(S)=\{g_1,\dots, g_t\}$ and $f$ is an odd integer greater than or equal to $3g(S)+1$, then it is shown that $T=\langle 2n_1, 2n_2,\dots,2n_k, f-2g_1,\dots,f-2g_t\rangle$ is a symmetric numerical semigroup with Frobenius number $f$ and $S=\frac{T}{2}$.

\medskip

For $s\in S=\langle n_1, n_2,\dots,n_k\rangle$, let $\Delta_s$ be the simplicial complex with faces $\{n_{i_1},\dots, n_{i_j}\}$ such that $s-(n_{i_1}+\cdots+n_{i_j})\in S$. 

The following proposition relates the graded Betti numbers of $K[S]$ with the dimension of the reduced homology of $\Delta_s$ over $K$.

\begin{prop}\cite[Lemma 1.1]{CM}\label{complex}
$\beta_{i,s}(K[S])=\dim_K\tilde{H}_{i-1}(\Delta_s,K)$.
\end{prop}
\begin{proof} As we wrote in the Introduction, $K[S]=K[X_1,\dots,X_k]/I$ and the Koszul complex on $X_1,\dots,X_k$ is a minimal resolution of $K$; tensoring with $K[S]$ and taking the $i$-th homology gives $\Tor_i^A(K[S],K)$. Then it is enough to note that the $s$-graded part of the Koszul complex tensored with $K[S]$ is isomorphic to the chain complex of $\Delta_s$ shifted by one.
\end{proof}

\section{Telescopic semigroups}

Due to the fact that telescopic semigroups have a nice structure, they have been intensively studied, \cite{KP}.
We are interested in computing the Betti numbers of $K[T]$, when $T$ is a symmetric numerical semigroup such that $S$ is its half and $S$ is a telescopic semigroup. 

\begin{dfn}\rm Let $(n_1,\ldots,n_k)$ be a sequence of positive integers with $n_1<n_2<\dots<n_k$ and such that their greatest common divisor is $1$. Define
$d_i=\gcd(n_1,\ldots,n_i)$ and $A_i=\{n_1/d_i,\ldots,n_i/d_i\}$ for $i =1,\ldots,k$. Let $S_i$ be the semigroup generated by $A_i$. If $n_i/d_i\in S_{i-1}$ for $i = 2,\ldots,k$, we call the sequence $(n_1,\ldots,n_k)$ \textit{telescopic}. A numerical semigroup is \textit{telescopic} if it is generated by a telescopic sequence.
\end{dfn}

\begin{exa}\rm
It is easy to check that $S=\langle 6,10,11\rangle$ is telescopic.
\end{exa}

\begin{rem}\rm
We note that the telescopic semigroups are symmetric (see \cite[Lemma 6.5]{KP} for instance).
\end{rem}

\begin{lemma}\label{a} 
Let $S=\langle n_1,\ldots,n_k\rangle$ be a telescopic semigroup and $f\geq 3g(S)+1$ an odd number. Then $T=\langle 2n_1,\ldots,2n_k,f-2g(S)\rangle$ is a telescopic semigroup.
\end{lemma}
\begin{proof} We have to show that $(2n_1,\ldots,2n_k,f-2g(S))$ is a telescopic sequence. The definition is verified by any subsequence of $(2n_1,\ldots,2n_k)$ since the sequence $(n_1,\ldots,n_k)$ is telescopic by hypothesis. Let $d_{k+1}=\gcd(2n_1,\ldots,2n_k,f-2g(S))$, then $d_{k+1}=1$. Since $T_{k}=S$ and $f-2g(S)\in T\subseteq \frac{T}{2}=S$, the statement follows.
\end{proof}

A special class of numerical semigroups is that of complete intersection. A numerical semigroup $S$ is a {\em complete intersection semigroup} if $K[S]$ is a complete intersection ring. For a purely numerical definition of $S$ complete intersection see \cite[page 129]{RGS}.


\begin{rem}\rm According to \cite[Lemma 1]{Wa}, the semigroup $T$ of Lemma~\ref{a} is a complete intersection semigroup.
\end{rem}

\begin{prop}\label{Bettitel} Let $S=\langle n_1,\ldots,n_k\rangle$ be a telescopic semigroup and $f\geq 3g(S)+1$ an odd number. Let $T_1=\langle 2n_1,\ldots,2n_k,f-2g(S)\rangle$ and $T_2=\langle 2n_1,\ldots,2n_k,f+2-2g(S)\rangle$. Then 
\[\beta_{1,j}(K[T_2])=
\left\{\begin {array}{ll}
			\beta_{1,j}(K[T_1]), & \mbox{ if}\ j\notin\{2(f-2g(S)),2(f-2g(S))+4\},\\
			0, & \mbox{ if}\ j=2(f-2g(S)),\\
			\beta_{1,j}(K[T_1])+1,& \mbox{ if}\ j=2(f-2g(S))+4.
	\end{array}\right. \]
\end{prop}

\begin{proof} Since $T_1$ and $T_2$ are complete intersection, by using the proof of \cite[Lemma 1]{Wa}, one has that 
\[\beta_{1,j}(K[T_1])=
\left\{\begin {array}{ll}
			1, & \mbox{ if}\ j\in\{2r_1,\ldots,2r_{k-1},2(f-2g(S))\},\\
			0, & \mbox{ otherwise}
	\end{array}\right. , \]

and

\[\beta_{1,j}(K[T_2])=
\left\{\begin {array}{ll}
			1, & \mbox{ if}\ j\in\{2r_1,\ldots,2r_{k-1},2(f+2-2g(S))\},\\
			0, & \mbox{ otherwise}
	\end{array}\right. , \]

where $r_1,\ldots,r_{k-1}$ are the degrees of the relations in $S$.



\end{proof}
 
 For simplicity, we denote the degrees $2r_i$ of the relations in $T_1$ (and in $T_2$), as in the previous proposition, by $N_{i}$ for $i=1,\ldots,k-1$ and $[k-1]:=\{1,\ldots,k-1\}$.

\begin{cor}\label{b} 
Let $S=\langle n_1,\ldots,n_k\rangle$ be a telescopic semigroup and $f\geq 3g(S)+1$ an odd number. Let $T_1=\langle 2n_1,\ldots,2n_k,f-2g(S)\rangle$ and $T_2=\langle 2n_1,\ldots,2n_k,f+2-2g(S)\rangle$. Then 
\[\beta_{ij}(K[T_2])=
\left\{\begin {array}{ll}
			\beta_{ij}(K[T_1]), & \mbox{if}\ j\in\{N_{t_1}+\cdots+N_{t_i}:\{t_1,\ldots,t_i\}\subseteq [k-1]\}\\ 
			0, & \mbox{if}\ j\in\{2(f-2g(S))+N_{t_1}+\cdots+N_{t_{i-1}}:\\
			& \{t_1,\ldots,t_{i-1}\}\subseteq [k-1]\},\\
			\beta_{ij}(K[T_1])+1,& \mbox{if}\ j\in\{2(f-2g(S))+4+N_{t_1}+\cdots+N_{t_{i-1}}:\\
			&\{t_1,\ldots,t_{i-1}\}\subseteq [k-1]\}.
	\end{array}\right. \]
\end{cor}
\begin{proof} The graded Betti number $\beta_{ij}$ of a semigroup ring associated to a numerical semigroup which is a complete intersection is obtained as a sum of $i$ different degrees of the first syzygy.
\end{proof}

\begin{rem}\label{c}
\rm We can apply the previous result to all the $3$-generated symmetric semigroups, since all of them are telescopic. Indeed any $3$-generated symmetric semigroup is of the kind $\langle dn_1,dn_2,a_1n_1+a_2n_2 \rangle$ where $\gcd(n_1,n_2)=1$, $a_1+a_2>1$ and $\gcd(d,a_1n_1+a_2n_2)=1$ (see \cite{He} and \cite{Wa}; or \cite[pages 10--11]{F}).
\end{rem}

\begin{exa}\rm We will determine the Betti numbers for the semigroups $T_1=\langle 12,20,22,205\rangle$ and
$T_2=\langle 12,20,22,207\rangle$. They are both equal to $\frac{S}{2}$, constructed from
$S=\langle 6,10,11\rangle$ with $f=255$ and $f=257$, respectively. We start with $U=\langle 3,5\rangle$. Then
$K[U]=K[X,Y]/(X^5-Y^3)$ with $\deg(X)=3$ and $\deg(Y)=5$, so $\deg(X^5-Y^3)=15$. Thus $\beta_{0,0}(K[U])=\beta_{1,15}(K[U])=1$
are the only nonzero Betti numbers. Then $S=\langle 2\cdot 3,2\cdot 5,11\rangle=\langle 6,10,11\rangle$ is also a complete intersection, and
$\beta_{0,0}(K[S])=\beta_{1,j}(K[S])=\beta_{2,k}=1$ for $j=2\cdot 11=22$, $j=2\cdot 15=30$, and $k=22+30=52$ are the only nonzero Betti numbers.

Finally
\[	\begin {array}{ll}
	\beta_{0,0}(T_1)=1& \\
	\beta_{1,j}(K[T_1])=1, &\mbox{for}\ j\in\{44=2\cdot 22,60=2\cdot 30,410=2\cdot 205\},\\
	\beta_{2,j}(K[T_1])=1, &\mbox{for}\ j\in\{104=44+60,454=44+410,470=60+410\},\\ 
	\beta_{3,j}(K[T_1])=1, &\mbox{for}\ j=514=44+60+410,
\end{array}\]
and $\beta_{i,j}(K[T_1])=0$ otherwise, and 

\[	\begin {array}{ll}
			\beta_{0,0}(K[T_2])=1 &\\
			\beta_{1,j}(K[T_2])=1, &\mbox{for}\ j\in\{44=2\cdot 22,60=2\cdot 30,414=2\cdot 207\},\\
			\beta_{2,j}(K[T_2])=1, &\mbox{for}\ j\in\{104=44+60,458=44+414,474=60+414\},\\ 
			\beta_{3,j}(K[T_2])=1, &\mbox{for}\ j=518=44+60+414,
			\end{array}\]
and $\beta_{i,j}(K[T_2])=0$ otherwise.
\end{exa}
\section{Semigroups with 3 generators which are not symmetric}

Let $S=\langle n_1,n_2,n_3\rangle$ be a numerical semigroup which is not symmetric, $n_1<n_2<n_3$. Since the type of a $3$-generated numerical semigroup is less than or equal to $2$ (see, \cite[Corollary 10.22]{RGS}) and the type is equal to one if and only if $S$ is symmetric, then the set $T(S)$ of pseudo-Frobenius numbers of $S$ is equal to $\{g_1,g_2\}$ where $g_1=g(S)$ is the Frobenius number of $S$. Let $c_i$ be the minimal positive integer such that $c_in_i\in\langle n_j,n_k\rangle$, $i\neq j,\neq k$ and consider $c_in_i=r_{i,j}n_j+r_{i,k}n_k$.

\begin{thm}(\cite{He}) Let $S=\langle n_1,n_2,n_3\rangle$ be a numerical semigroup which is not symmetric. Then, in the above notation,
\[ \beta_{1,j}(S)=
\left\{\begin {array}{ll}
			1, & \mbox{ if}\ j\in\{n_1c_1,n_2c_2,n_3c_3\},\\
			0, & \mbox{ otherwise},
	\end{array}\right. \] 
	\[ \beta_{2,j}(S)=
\left\{\begin {array}{ll}
			1, & \mbox{ if}\ j\in\{n_2c_2+n_3r_{1,3},n_3c_3+n_2r_{1,2}\},\\
			0, & \mbox{ otherwise}
	\end{array}\right. \] 
and $\beta_i(S)=0$, for $i\geq3$.
\end{thm} 

Let now $\varphi:K[X,Y,Z]\longrightarrow K[S]$ be the $K$-algebra homomorphism defined by $\varphi(X)=t^{n_1}$, $\varphi(Y)=t^{n_2}$ and $\varphi(Z)=t^{n_3}$ and let $I_S=\ker\varphi$. 

\begin{prop}\cite{He}\label{1aa}
The ideal $I_S$ is generated by the maximal minors of the matrix 
$$
M=\begin{pmatrix}X^a & Y^b & Z^c\\ Z^d & X^e & Y^f
\end{pmatrix}
$$

\noindent for some $a,b,c,d,e,f\in\mathbb N_{>0}$. 
\end{prop} 

We can relate the exponents of the variables in the matrix $M$ with the $r_{i,j}$.

\begin{prop}\label{1a}
The ideal $I_S$ is generated by the maximal minors of the matrix
 
$$
\begin{pmatrix}X^{r_{3,1}} & Y^{r_{1,2}} & Z^{r_{2,3}}\\ Z^{r_{1,3}} & X^{r_{2,1}} & Y^{r_{3,2}}
\end{pmatrix}.
$$
\end{prop}
 
\begin{proof} Since $K[S]/(t^{n_1})\cong K[Y,Z]/(Y^bZ^d,Z^{c+d},Y^{b+f})$ and $n_1=\dim K[S]/(t^{n_1})$, and likewise for $n_2,n_3$, we get the equations

$$
n_1=fd+bd+bc,
$$
$$
n_2=ed+ec+ac,
$$
$$
n_3=ef+af+ab.
$$

Hence we have 

$$
(a+e)n_1=bn_2+dn_3,
$$
$$
\ \ (f+b)n_2=en_1+cn_3,\ \ \ \ (1)
$$
$$
(c+d)n_3=an_1+fn_3
$$

\noindent with $a+e=c_1$, $f+b=c_2$ and $c+d=c_3$ (see, \cite[prop 2.1]{NNW}). Finally we get 

$$
M=\begin{pmatrix}X^{r_{3,1}} & Y^{r_{1,2}} & Z^{r_{2,3}}\\ Z^{r_{1,3}} & X^{r_{2,1}} & Y^{r_{3,2}}
\end{pmatrix}.
$$
\end{proof}

We recall that $T(S)=\{g_1,g_2\}$ is the set of pseudo-Frobenius numbers of $S$, where $g_1=g(S)$ is the Frobenius number of $S$. 
Our next aim is to express the difference $g_1-g_2$ in terms of the $r_{i,j}$ using the fact that, for numerical semigroups with embedding dimension three, the set of pseudo-Frobenius numbers is known:

\begin{prop}\label{T(S)}\cite[Corollary 12]{RG} Let $S=\langle n_1,n_2,n_3\rangle$ be a numerical semigroup. Then $$T(S)=\{(c_3-1)n_3+(r_{1,2}-1)n_2-n_1,\ (c_2-1)n_2+(r_{1,3}-1)n_3-n_1\}.$$
\end{prop} 

The generators of the numerical semigroup $S$ can be expressed in terms of the numbers $r_{i,j}$:

\begin{lemma}\cite[Lemma 5]{RG}\label{4}
$$
n_1=r_{1,2}r_{1,3}+r_{1,2}r_{2,3}+r_{1,3}r_{3,2},
$$
$$
n_2=r_{1,3}r_{2,1}+r_{2,1}r_{2,3}+r_{2,3}r_{3,1},
$$
$$
n_3=r_{1,2}r_{3,1}+r_{2,1}r_{3,2}+r_{3,1}r_{3,2}.
$$
\end{lemma}

Using the above results, we get:

\begin{cor}\label{3} In the above notation
\begin{itemize}
	\item[(i)]$g_1-g_2=r_{1,3}r_{2,1}r_{3,2}-r_{1,2}r_{2,3}r_{3,1}$, if $g_1=(c_2-1)n_2+(r_{13}-1)n_3-n_1$.
	\item[(ii)]$g_1-g_2=r_{1,2}r_{2,3}r_{3,1}-r_{1,3}r_{2,1}r_{3,2}$, if $g_1=(c_3-1)n_3+(r_{12}-1)n_2-n_1$.
\end{itemize}
\end{cor} 

\begin{proof} The proof follows immediately by Proposition~\ref{T(S)} and Lemma~\ref{4}.  
\end{proof}

\begin{cor}\label{5}
Let $r_{1,2},r_{1,3},r_{2,1},r_{2,3},r_{3,1},r_{3,2},n_1,n_2,n_3$ be as above. 
\begin{itemize}
	\item[(i)] If $g_1=(c_2-1)n_2+(r_{13}-1)n_3-n_1$, then $r_{1,3}n_3-r_{3,1}n_1=r_{2,1}n_1-r_{1,2}n_2$ $=r_{3,2}n_2-r_{2,3}n_3=g_1-g_2.$
\item[(ii)] If $g_1=(c_3-1)n_3+(r_{12}-1)n_2-n_1$, then $r_{3,1}n_1-r_{1,3}n_3=r_{1,2}n_2-r_{2,1}n_1$ $=r_{2,3}n_3-r_{3,2}n_2=g_1-g_2.$
\end{itemize}
\end{cor}
\begin{proof} The proof follows immediately by Proposition~\ref{T(S)}, Lemma~\ref{4} and Corollary~\ref{3}.
\end{proof}

Let $S=\langle n_1,n_2,n_3\rangle$ be a numerical semigroup, $f\geq 3g(S)+1$ be an odd number and $T=\langle 2n_1,2n_2,2n_3,f-2g_1,f-2g_2\rangle$. Furthermore, let us consider $\psi:K[X,Y,Z,U,V]\rightarrow K[T]$ the $K$-algebra homomorphism defined by $\psi(X)=t^{2n_1}$, $\psi(Y)=t^{2n_2}$, $\psi(Z)=t^{2n_3}$, $\psi(U)=t^{f-2g_1}$ and $\psi(V)=t^{f-2g_2}$ and let $I_T=\ker\psi$.

\begin{prop}\label{7,8,9} In the above notation, one has that $U^2-m_1, UV-m_2, V^2-m_3$ are generators of $I_T$, where $m_1,m_2,$ and $m_3$ are monomials in $K[X,Y,Z]$.
\end{prop}

\begin{proof} Since $f-2g_1\in T\subseteq \frac{T}{2}=S$, one has that $2(f-2g_1)\in 2S$, therefore $2(f-2g_1)=\alpha_1(2n_1)+\alpha_2(2n_2)+\alpha_3(2n_3)$, with $\alpha_1,\alpha_2,\alpha_3\in\mathbb{N}$ and $U^2-m_1\in I_T$, where $m_1=X^{\alpha_1}Y^{\alpha_2}Z^{\alpha_3}$. Similarly, one obtains that $UV-m_2$ and $V^2-m_3$ are in $I_T$, for some monomials $m_2$ and $m_3$ in $K[X,Y,Z]$.
\end{proof}

\medskip 

\noindent In order to look for more relations we can compute modulo $(U^2-m_1,UV-m_2,V^2-m_3)$. There are no relations of the form $U-m$ or $V-m$, $m\in K[X,Y,X]$, because $U$ and $V$ are of odd degree and $m$ of even degree. Then we look for relations of the form $f(X,Y,Z)U-g(X,Y,Z)V$ and for relations $f(X,Y,Z)=0$ (the latter ones are exactly the same as the relations in $K[S]$). All this is done in the next two propositions.


\begin{prop}\label{6} Let $g_1=(c_2-1)n_2+(r_{1,3}-1)n_3-n_1$. Then the maximal minors of the matrix

$$
N=\begin{pmatrix}X^{r_{3,1}} & Y^{r_{1,2}} & Z^{r_{2,3}} & U\\ Z^{r_{1,3}} & X^{r_{2,1}} & Y^{r_{3,2}} & V
\end{pmatrix}
$$

\noindent are a subset of minimal generators of $I_T$.
\end{prop}
\begin{proof} Using Proposition \ref{1aa}, we know that the minors of $N$ obtained by considering only the first three columns are a subset of minimal generators of $I_T$. 

We are looking for elements in $I_T$ of the kind $f(X,Y,Z)U-g(X,Y,Z)V$, that is for

$$
\det\begin{pmatrix}f & V\\ g & U
\end{pmatrix}.
$$
Since $\deg(V)-\deg(U)=2(g_1-g_2)$, then we must have $\deg(f)-\deg(g)=2(g_1-g_2)$. Thus we are looking for columns
$$
 \begin{pmatrix}f\\ g
\end{pmatrix}\in K[X,Y,Z]^2
$$
\noindent such that $\deg(f)-\deg(g)=2(g_1-g_2)$. Let $W$ be the $K[X,Y,Z]$-submodule of such columns. We need to find a minimal set of generators for $W$. But, from Corollary~\ref{5} and Propositions \ref{1aa} and \ref{1a}, we get that the minimal generators for $W$ are the $2\times 2$-minors of the matrix    

$$
M=\begin{pmatrix}X^{r_{3,1}} & Y^{r_{1,2}} & Z^{r_{2,3}}\\ Z^{r_{1,3}} & X^{r_{2,1}} & Y^{r_{3,2}}
\end{pmatrix}.
$$

This means that the minimal generators of $W$ are the columns of $M$.
\end{proof}
\rm As before, we get a similar result for the case when $g_1=(c_2-1)n_2+(r_{13}-1)n_3-n_1$:

\begin{prop}\label{6'} Let $g_1=(c_3-1)n_3+(r_{1,2}-1)n_2-n_1$. Then the maximal minors of the matrix

$$
N=\begin{pmatrix}X^{r_{3,1}} & Y^{r_{1,2}} & Z^{r_{2,3}} & V\\ Z^{r_{1,3}} & X^{r_{2,1}} & Y^{r_{3,2}} & U
\end{pmatrix}
$$

\noindent are a subset of minimal generators of $I_T$.
\end{prop}

\noindent By what is written above in the section we soon get the following proposition.

\begin{prop}\label{7}
If $g_1=(c_2-1)n_2+(r_{1,3}-1)n_3-n_1$, then $I_T$ is minimaly generated by $\{X^{r_{3,1}+r_{2,1}}-Y^{r_{1,2}}Z^{r_{1,3}},\ X^{r_{3,1}}Y^{r_{3,2}}-Z^{r_{1,3}+r_{2,3}},\ Y^{r_{1,2}+r_{3,2}}-X^{r_{2,1}}Z^{r_{2,3}},\ X^{r_{3,1}}V$ $-Z^{r_{1,3}}U,\ Y^{r_{1,2}}V-X^{r_{2,1}}U,\  Z^{r_{2,3}}V-Y^{r_{3,2}}U,\ U^2-m_1, UV-m_2,\ V^2-m_3\}$, with $m_1$, $m_2$ and $m_3$ monomials in $K[X,Y,Z]$. If $g_1=(c_3-1)n_3+(r_{12}-1)n_2-n_1$, then $I_T$ is minimaly generated by the set $\{X^{r_{3,1}+r_{2,1}}-Y^{r_{1,2}}Z^{r_{1,3}},\ X^{r_{3,1}}Y^{r_{3,2}}-Z^{r_{1,3}+r_{2,3}},\ Y^{r_{1,2}+r_{3,2}}-X^{r_{2,1}}Z^{r_{2,3}}, \ X^{r_{3,1}}U-Z^{r_{1,3}}V,\ Y^{r_{1,2}}U-X^{r_{2,1}}V,\ Z^{r_{2,3}}U-Y^{r_{3,2}}V,\ $ $U^2-m_1,\ UV-m_2,\ V^2-m_3\}$, with $m_1$, $m_2$ and $m_3$ monomials in $K[X,Y,Z]$. 
\end{prop}



\medskip 

In the next proposition we get the Betti numbers $\beta_{i,j}(K[T])$ for $i\ne 2$. For simplicity, we denote
 
$$
B_1=\{2n_2r_{1,2}+2n_3r_{1,3},\ 2n_1r_{3,1}+2n_2r_{3,2},\ 2n_1r_{2,1}+2n_3r_{2,3},
\ 2n_3r_{1,3}+(f-2g_1),
$$
$$\ \ \ \ \ \ \ \ \ \ \ \ \ 2n_1r_{2,1}+(f-2g_1),\ 2n_2r_{3,2}+(f-2g_1),2(f-2g_1),\ (f-2g_1)+(f-2g_2),$$
$$\ 2(f-2g_2)\},\ \ \ \ \ \ \ \ \ \ \ \ \ \ \ \ \ \ \ \ \ \ \ \ \ \ \ \ \ \ \ \ \ \ \ \ \ \ \ \ \ \ \ \ \ \ \ \ \ \ \ \ \ \ \ \ \ \ \ \ \ \ \ \ \ \ \ \ \ \ \ \ \ \ \ \ \ \ \ \ \ \ \ \ \ \ \ \ 
$$

$$
B_2=\{2n_2r_{1,2}+2n_3r_{1,3},\ 2n_1r_{3,1}+2n_2r_{3,2},\ 2n_1r_{2,1}+2n_3r_{2,3},\ 2n_3r_{2,3}+(f-2g_1),
$$
$$
\ \ \ \ \ \ \ \ \ \ \ \ \ 2n_1r_{3,1}+(f-2g_1),\ 2n_2r_{1,2}+(f-2g_1), 2(f-2g_1), (f-2g_1)+(f-2g_2),$$
$$ 2(f-2g_2)\},\ \ \ \ \ \ \ \ \ \ \ \ \ \ \ \ \ \ \ \ \ \ \ \ \ \ \ \ \ \ \ \ \ \ \ \ \ \ \ \ \ \ \ \ \ \ \ \ \ \ \ \ \ \ \ \ \ \ \ \ \ \ \ \ \ \ \ \ \ \ \ \ \ \ \ \ \ \ \ \ \ \ \ \ \ \
$$

\medskip

\noindent and with $\alpha=2(n_1+n_2+n_3)+(f-2g_1)+(f-2g_2)+f$. Moreover we denote by $B'_i$ the sets $\alpha-B_i:=\{\alpha-b\ |\ b\in B_i\}$ where $i=1,2$. 

\begin{prop}\label{8}
Let $S=\langle n_1,n_2,n_3\rangle$ be a numerical semigroup which is not symmetric, $T(S)=\{g_1,g_2\}$ be the set of pseudo-Frobenius numbers of $S$ (where $g_1=g(S)$), $f\geq 3g(S)+1$ be an odd number and $T=\langle 2n_1,2n_2,2n_3,f-2g_1,f-2g_2\rangle$. Then 

\[\beta_{0,j}(K[T])=
\left\{\begin {array}{ll}
			1, & \mbox{ if}\ j=0,\\
			0, & \mbox{ otherwise},
	\end{array}\right. \] 

\[\beta_{4,j}(K[T])=
\left\{\begin {array}{ll}
			1, & \mbox{ if}\ j=\alpha,\\
			0, & \mbox{ otherwise}
	\end{array}\right. \] 
and $\beta_{i,j}(K[T])=0$ for every $i\ge 5$ and every $j$. 

\noindent Moreover, if $g_1=(c_2-1)n_2+(r_{1,3}-1)n_3-n_1$, then 
\[\beta_{1,j}(K[T])=
\left\{\begin {array}{ll}
			1, & \mbox{ if}\ j\in B_1,\\
			0, & \mbox{ otherwise},
	\end{array}\right. \]
and

\[\beta_{3,j}(K[T])=
\left\{\begin {array}{ll}
			1, & \mbox{ if}\ j\in B'_1,\\
			0, & \mbox{ otherwise}.
	\end{array}\right. \] 
Otherwise, if $g_1=(c_3-1)n_3+(r_{1,2}-1)n_2-n_1$, then 
\[\beta_{1,j}(K[T])=
\left\{\begin {array}{ll}
			1, & \mbox{ if}\ j\in B_2,\\
			0, & \mbox{ otherwise},
	\end{array}\right. \] 

\[\beta_{3,j}(K[T])=
\left\{\begin {array}{ll}
			1, & \mbox{ if}\ j\in B'_2,\\
			0, & \mbox{ otherwise}.
	\end{array}\right. \] 

\begin{proof} For the zero Betti numbers there is nothing to prove. 

By $\beta_{4,j}=\dim_K\tilde{H}_{3}(\Delta_j)$, the only possibility for $\beta_{4,j}$ to be different from zero is for $\beta_{4,j}=1$ and $\Delta_j$ being the empty solid with vertices $\{2n_1,2n_2,2n_3,f-2g_1,f-2g_2\}$, that is $j-(2n_1+2n_2+2n_3+(f-2g_1)+(f-2g_2))$ must not be in $T$  while each time we subtract the sum of $4$ different generators of $T$ to $j$ the result must be in $T$. This can happen only for $j=\alpha$, reminding that $f$ is the Frobenius number of $T$.

The formula $\beta_{i,j}=\dim_K\tilde{H}_{i-1}(\Delta_j)$ explains why $\beta_{i,j}(K[T])=0$ for every $i\ge 5$ and every $j$.

The results for the first Betti numbers follow by Proposition \ref{7}. 

For the third Betti numbers, we note that, since $T$ is a symmetric numerical semigroup, we have $\beta_{3,j}+\beta_{1,\alpha-j}=\beta_{4,\alpha}$ for every $j$.  
\end{proof}
\end{prop}
\medskip 

For simplicity of notation, we denote by
 
$$
M=\{2n_2r_{1,2}+2n_3r_{1,3}, 2n_1r_{3,1}+2n_2r_{3,2}, 2n_1r_{2,1}+2n_3r_{2,3}\}
$$

\noindent and 

$$
P=\{2(f-2g_1), (f-2g_1)+(f-2g_2), 2(f-2g_2)\}
$$ 

\noindent the subsets of $B_1\cap B_2$, and by

$$
N_1=\{2n_3r_{1,3}+(f-2g_1), 2n_1r_{2,1}+(f-2g_1), 2n_2r_{3,2}+(f-2g_1)\}
$$

\noindent and 

$$
N_2=\{2n_3r_{2,3}+(f-2g_1), 2n_1r_{3,1}+(f-2g_1), 2n_2r_{1,2}+(f-2g_1)\}
$$

\noindent the subsets of $B_1$ and $B_2$ respectively.

Using the same terminology as above, we denote by $A'$ the set $\alpha-A:=\{\alpha-a\ |\ a\in A\}$ and by $A+n$ the set $\{a+n\ |\ a\in A,\ n\in\mathbb{N}\}$.

\begin{prop}\label{9}
Let $S=\langle n_1,n_2,n_3\rangle$ be numerical semigroup which is not symmetric, $T(S)=\{g_1,g_2\}$ be the set of pseudo-Frobenius numbers of $S$ (where $g_1=g(S)$) and $f\geq 3g(S)+1$ be an odd number. Let $T_1=\langle 2n_1,2n_2,2n_3,f-2g_1,f-2g_2\rangle$ and $T_2=\langle 2n_1,2n_2,2n_3,f-2g_1+2,f-2g_2+2\rangle$. Finally, assume that in what follows $i\ne 2$. Then, if $g_1=(c_2-1)n_2+(r_{1,3}-1)n_3-n_1$, we have 
\[\beta_{i,j}(K[T_2])=
\left\{\begin {array}{ll}
			\beta_{i,j}(K[T_1]), & \mbox{ if}\ j\in M,\\
			\beta_{i,j}(K[T_1])+1, & \mbox{ if}\ \in (N_1+2)\cup (P+4)\cup (M'+6)\cup \\
			 & \ \cup(N'_1+4)\cup \cup(P'+2)\cup (\{\alpha\}+6),\\
			0, & \mbox{ otherwise}.
	\end{array}\right. \]
Moreover, if $g_1=(c_3-1)n_3+(r_{1,2}-1)n_2-n_1$, then
\[\beta_{i,j}(K[T_2])=
\left\{\begin {array}{ll}
			\beta_{i,j}(K[T_1]), & \mbox{ if}\ j\in M,\\
			\beta_{i,j}(K[T_1])+1, & \mbox{ if}\ \in (N_2+2)\cup (P+4)\cup (M'+6)\cup\\
			& \ \cup (N'_2+4)\cup(P'+2)\cup (\{\alpha\}+6),\\
			0, & \mbox{ otherwise}.
	\end{array}\right. \]
\end{prop}
\begin{proof} This follows easily by Proposition~\ref{8} and by definitions of $T_1$ and $T_2$.

\end{proof}
\begin{rem}\rm For each concrete example we can also determine $\beta_{2,j}(K[T])$. We
give an example. Let $S=\langle 3,5,7\rangle$ and $T=\langle 6,10,14,15\rangle$. Then $S=\frac{T}{2}$. We use that
the Hilbert series of $K[T]$ is
$$H(K[T],t)=\frac{\sum(-1)^i\beta_{i,j}t^j}{(1-t^6)(1-t^{10})(1-t^{14})(1-t^{15})}.$$
Since we know by Proposition \ref{9} that $\beta_{1,j}=1$ for $j=20,24,28,30$ and $\beta_{3,j}=1$ for
$j=64,68$, we can compute $\beta_{2,j}$ from the Hilbert series of $K[T]$. Indeed, we may also write the
Hilbert series of $k[T]$ as
$$H(K[T],t)=1+t^6+t^{10}+t^{12}+t^{14}+t^{15}+t^{16}+t^{18}+t^{20}+t^{21}+t^{22}+
\frac{t^{24}}{1-t}.$$ A short calculation gives that $\beta_{2,j}=1$ for
$j=34,38,50,54,58$.
\end{rem}
\section{$4$-generated symmetric semigroups}
Bresinsky shows in \cite{B} that a semigroup ring associated to a $4$-generated symmetric semigroup that is not
a complete intersection has $\beta_1=5$. This gives the following proposition.

\begin{prop}\label{4sym} If $S$ is a $4$-generated symmetric semigroup which is not a complete
intersection, and $T=\frac{S}{2}$, then $\beta_1(K[T])=6$, $\beta_2(K[T])=10$, $\beta_3(K[T])=6$,
and $\beta_4(K[T])=1$. In particular all $T$ have the same total Betti numbers.
\end{prop}
\begin{proof} Let $S=\langle n_1,n_2,n_3,n_4\rangle$ be symmetric and not a complete
intersection. Then we have the numerical semigroup $T=\langle 2n_1,2n_2,2n_3,2n_4, f-2g\rangle$ with $g=g(S)$. Let  $\psi:K[X_1,X_2,\ldots,X_5]\rightarrow K[T]$ be the $K$-algebra homomorphism defined by $\psi(X_i)=t^{2n_i}$ for $i=1,2,3,4$ and $\psi(X_5)=t^{f-2g}$ and let $I_T=\ker\psi$.
By definitions of $g(S)$ and $f$, we get $X_5^2-f(X_1,\ldots,X_4)\in I_T$. If we compute modulo $X_5^2-f(X_1,\ldots,X_4)$, then we cannot have a relation  $X_5-g(X_1,\ldots,X_4)$ for degree reason and the other relations $X_5f(X_1,\ldots,X_4)-X_5g(X_1,\ldots,X_4)$ follows from the relations $f(X_1,\ldots,X_4)-g(X_1,\ldots,X_4)$ and these are the old relations coming from $S$. Since $\beta_1(K[S])=5$, then $\beta_1(K[T])=6$. Now, since $T$ is symmetric, then $K[T]$ is Gorenstein, so the Betti numbers are symmetric. We have $\beta_4(K[T])=1$ and (from the symmetry) $\beta_3(K[T])=6$. The alternating sum of the Betti numbers is $0$, so $\beta_2(K[T])=10$.
\end{proof}

\begin{rem}\rm As we note above, Bresinsky shows in \cite{B} that a semigroup ring associated to a $4$-generated symmetric semigroup that is not
a complete intersection has $\beta_1=5$. The corresponding is not true for $5$-generated semigroups. Indeed, all the semigroups $T$ in Proposition~\ref{4sym} are $5$-generated symmetric and not complete intersection with $\beta_1(K[T])=6$ but all the semigroups $T$ in Section 4 are $5$-generated symmetric and not complete intersection with $\beta_1(K[T])=9$ (see, Proposition \ref{8}).
\end{rem}

\section{4--irreducible semigroups}

A numerical semigroup is called \textit{irreducible} if it cannot be expressed as an intersection of two numerical semigroups containing it properly. It is known that the family of irreducible numerical semigroups is the union of symmetric and pseudo-symmetric semigroups, \cite{BR}.

In \cite{BR} it is defined the notion of \textit{$m$-irreducibility} which extends the concept of irreducibility when the multiplicity is fixed. More precisely, a numerical semigroup of multiplicity $m$ is called \textit{$m$-irreducible} if it cannot be written as an intersection of two numerical semigroups with multiplicity $m$ properly containing it. In \cite{BR}, the set of $m$-irreducible semigroups is characterized.

\begin{prop}\cite[Proposition 6]{BR} A numerical semigroup $S$ with multiplicity $m$ is $m$-irreducible if and only if one of the following conditions holds:
\begin{enumerate}
	\item $S=\{x\in\mathbb{N}\ :\ x\geq m\}\cup\{0\}.$
	\item $S=\{x\in\mathbb{N}\ :\ x\geq m,\ x\neq F(S)\}\cup\{0\}.$
	\item $S$ is an irreducible numerical semigroup with multiplicity $m$.
\end{enumerate}
\end{prop}

For the case of numerical semigroups of multiplicity $4$, the above proposition can be stated as follows:
\begin{cor} A numerical semigroup $S$ of multiplicity $4$ is $4$-irreducible if and only if $S\in\{\langle 4,5,6,7\rangle,\langle 4,6,7,9\rangle\}$ or $S$ is irreducible.
\end{cor}

\begin{rem}\label{4irr}
\rm We aim at determining the Betti numbers of the semigroup ring $K[S]$ when $S$ is $4$-irreducible. If $S\in\{\langle 4,5,6,7\rangle,\langle 4,6,7,9\rangle\}$, then $\beta_{i,j}(K[S])$ can be easily determined by using any computer algebra system. 

For the case of irreducible semigroups, one has that they are minimally generated by at most $3$ elements, since they are either symmetric or pseudo-symmetric, so they cannot be of maximal embedding dimension since $4$-generated semigroups of maximal embedding dimension has $|T(S)|=3$ and irreducible have $|T(S)|\le 2$. Since all $3$-generated symmetric numerical semigroups are complete intersection, then the Betti numbers are known as soon as the first Betti numbers are known (indeed $\beta_{i,j}(K[S])$ of $S$ which is a complete intersection is obtained as a sum of $i$ different degrees of the first syzygy) and this is done in \cite{Wa} (see also the proof of Corollary~\ref{b}). For the case of 3-generated pseudo-symmetric semigroups, the Betti numbers are also known (see \cite{He}).
\end{rem}

\section{Acknowledgements}

The authors wish to thank Professors Mats Boij and Ralf Fr\"oberg for the valuable conversations concerning this paper. We
also wish to thank Professor Alfio Ragusa and all other
organizers of PRAGMATIC 2011 for the opportunity to participate and for the pleasant atmosphere they
provided during the summer school.

\end{document}